\documentclass[leqno]{article}

\usepackage{amsmath,amssymb,amsthm}
\usepackage[latin1]{inputenc}
\usepackage{amscd}
\usepackage{graphics}

\theoremstyle{plain}
  \newtheorem{theorem}{Theorem}
  
  \newtheorem{lemma}{Lemma}

\theoremstyle{definition}
  \newtheorem{definition}{Definition}
  \newtheorem{remark}{Remark}

\makeatletter
\renewcommand{\section}{\@startsection
{section}%
{1}%
{0em
}%
{10mm}%
{-4mm}%
{\bfseries\normalsize}}%
\makeatother

\def\N{\mathbb N}
\def\C{\mathbb C}
\def\R{\mathbb R}

\begin{document}

\author{L. Bernal-Gonz\'alez  and A. Bonilla }

\date{}

\title{Rate of growth of hypercyclic and frequently hypercyclic functions for the Dunkl operator}

\maketitle

\begin{abstract} 
For the Dunkl operator $\Lambda_\alpha$ $(\alpha > -1/2)$ on the space of entire functions on the complex space $\C$, the
critical rate of growth for the integral means $M_p(f,r)$ of their hypercyclic functions $f$ is obtained. The rate of growth of the
corresponding frequently hypercyclic functions is also analyzed.
\end{abstract}


\section{Introduction.}\label{S-intro}

In this paper, we consider the vector space \,$H(\C )$ \,of all entire functions $\C \to \C$,
endowed with the topology of local uniform convergence.
Under this topology, $H(\C )$ becomes an F-space, that is, a completely metrizable topological vector space.
Recall that a (continuous and linear) operator $T$ on a (Hausdorff) topological vector
space $X$ is said to be {\it hypercyclic} if there exists a vector
$x\in X$, also called {\it hypercyclic}, whose orbit $\{T^nx:n\in
\mathbb{N}\}$ is dense in $X$. We refer the reader to the books \cite{BaMa09} and \cite{GrPe11} for
rather complete accounts and further information on hypercyclic operators.

\vskip .15cm

In 1952, MacLane \cite{Mac52} proved that the differentiation operator $D:H(\mathbb{C})\to H(\mathbb{C})$, given by
\[
Df(z)=f'(z),
\]
is hypercyclic. Moreover, he showed that $D$-hypercyclic
entire functions can be of exponential type 1. In 1984, Duyos-Ruiz
\cite{Duy84} they cannot be of exponential type less than 1. An
optimal result on the possible rates of growth for the
differentiation operator was subsequently obtained by Grosse-Erdmann
\cite{Gro90} and, independently, by Shkarin \cite{Shk93}. Specifically, they obtained that there is no $D$-hypercyclic entire function $f$ for which there exists a constant $C > 0$ satisfying $|f(z)| \le C \, {e^r \over \sqrt{r}}$ for $|z| = r$ large enough but, given a function $\varphi : \R_+ \to \R_+$ with $\varphi (r) \to \infty$ as $r \to \infty$, there exists a $D$-hypercyclic entire function such that $|f(z)| \le \varphi (r) \, {e^r \over \sqrt{r}}$ for $|z| = r$ large enough.

\vskip .15cm

The notion of frequent hypercyclicity that was recently introduced
by Bayart and Grivaux \cite{BaGr04}, \cite{BaGr06}. We recall that the lower density of a subset $A$ of
$\mathbb{N}$ is defined as
\[
\underline{\mbox{dens}}\,(A) = \liminf_{N\to\infty}\frac{\#\{n\in A : n\leq N\}}{N},
\]
where $\#$ denotes the cardinality of a set.

\vskip .15cm

A vector $x \in X$ is called {\it frequently hypercyclic} for
$T$ if, for every non-empty open subset $U$ of $X$,
\[
\underline{\mbox{dens}}\,\{n\in \mathbb{N} : T^n x \in U\} > 0.
\]
The operator $T$ is called {\it frequently hypercyclic} if it
possesses a frequently hypercyclic vector.

\vskip .15cm

The problem of determining possible rates of growth of
frequently hypercyclic entire functions for differentiation operator was studied in
\cite{BlaBoGro}, \cite{BoBo13} and \cite{drasinsaksman2012}. More generally, the problem of isolating the possible rates of growth of
hypercyclic or frequently hypercyclic entire functions for convolution operators on $H(\Bbb C)$
(operators that conmute with the differentiation operator) has been considered in
\cite{BeBo02},  \cite{BoGE06} and  \cite{ChSh91}. Finally, the rates of growth of hypercyclic entire functions for weighted backward shifts $B_w$ on $H(\C )$ --considered as a special sequence space-- have been established in \cite{Gro00}.
More classes of hypercyclic non-convolution operators on $H(\C )$ --different from composition operators-- have been analyzed by a number of
authors, see e.g.~\cite{aronmarkose2004,fernandezhallack2005,Kim1,Kim2,leonromero2014,petersson2005,petersson2005b,petersson2006}.

\vskip .15cm

One remarkable example is the {\it Dunkl operator}
$$
\Lambda _{\alpha} : H(\C ) \to H(\C ) \quad (\alpha > - \frac{1}{2}),
$$
which is a differential-difference operator given by
\begin{equation*}\label{Dunkl}
\Lambda _{\alpha}f(z) = \frac{d}{dz}f(z) + \frac{2\alpha +1}{z}\Big(\frac{f(z) - f(-z)}{2}\Big ).
\end{equation*}
Note that for \,$\alpha = - \frac{1}{2}$ \,we get \,$\Lambda _{\alpha} = D$.
The operator \,$\Lambda _{\alpha}$ \,was introduced in 1989 by Dunkl \cite{dunkl1989}.
It is connected to the theory of sampling signals, and there are in the literature a lot of papers dealing with the Dunkl operator,
see for instance \cite{ciaurrivarona2008} and the references given in it.

\vskip .15cm

Here we will concentrate on the dynamical aspects of \,$\Lambda_\alpha$.
The hypercyclicity of this operator was established by Betancor, Sifi and Trimeche in \cite{BeSiTri} (see also \cite{KimNa}and \cite{ChMeMiTri}).
The purpose of this note is to study the rate of growth of hypercyclic and frequently hypercyclic entire functions for the Dunkl operator.
In fact, the critical rate of growth for the $p$-integral means $M_p(f,r)$ of their hypercyclic functions $f$ is obtained.
Moreover, we get permissible and non-permissible rates of growth for the means of the corresponding frequently hypercyclic functions.

\section{Preliminaries and notation.} \label{Section-preliminaries}

As usual, throughout this paper constants $C>0$ can take different values at different occurrences.
We write $a_n \sim b_n$ for positive sequences $(a_n)$ and $(b_n$) if $a_n/b_n$ and $b_n/a_n$ are bounded.

\vskip .15cm

For an entire function $f$ and $1 \leq p < \infty$ we consider the $p$-integral means
\[
M_p(f,r)=\Big(\frac{1}{2\pi}\int_0^{2\pi} |f(re^{it})|^p
dt\Big)^{1/p} \quad (r>0)
\]
and
\[
M_\infty(f,r)=\sup_{|z|=r} |f(z)| \quad (r > 0).
\]

These means have been considered by Blasco, Bonet, Bonilla and Grosse-Erdmann (see \cite{BlaBoGro,BoBo13}) in order to establish the
possible rates of growth of hypercyclic or frequently hypercyclic entire functions for the operator \,$D$. Some of the techniques from \cite{BlaBoGro,BoBo13,Gro90} will be used and generalized in this paper in order to analyze the corresponding problem for $\Lambda_{\alpha}$.

\vskip .15cm

With this aim, some properties of the Dunkl operators are needed. Let us define
\begin{equation*}\label{dn}
d_{n}(\alpha) = {2^n \, \big([\frac{n}{2}]!\big) \over \Gamma(\alpha+1)} \, \Gamma \left([\frac{n+1}{2}] +\alpha +1 \right) \hbox{ \ for } n\ge 0
\hbox{ \,and\, } \alpha > - \frac{1}{2},
\end{equation*}
and $d_{n}(\alpha ) = 0$ when $n<0$, where $[x]$ denotes the integer part of $x$.
Notice that $d_{n}(\alpha )$ tends rapidly to $+\infty$ as $n \to \infty$.
Then, for every $k\in \Bbb N$, we have (see e.g.~\cite[pag. 106]{BeSiTri}) that
\begin{equation*}\label{zeta}
\Lambda ^{k} _{\alpha}(z^{n})= {d_n(\alpha) \over d_{n-k}(\alpha)} \, z^{n-k} \hbox{ \ for all }z\in \Bbb C .
\end{equation*}
Moreover, given $f(z)= \sum_{n=0}^{\infty} \frac{f^{(n)}(0)}{n!}z^{n} \in H(\C )$, the following holds:
$$
\Lambda_{\alpha}^{n}f(0) = \frac{f^{(n)}(0)}{n!} \, d_{n}(\alpha).
$$
By using Stirling's asymptotic formula (see e.g.~\cite{ahlfors1979}) \,$\Gamma (x) = (2\pi )^{1/2} \, x^{x - {1 \over 2}} \, e^{-x} \, \Psi (x)$ \,for $x > 0$ (where $\Psi (x) \to 1$ as $x \to +\infty$) and the facts \,$\Gamma (n+1) = n!$, $[{n \over 2}] + [{n+1 \over 2}] = n$  $(n \ge 0)$, one can easily obtain the following lemma, which will be used in the forthcoming sections.

\begin{lemma} \label{Lemma-Stirling}
For each $\alpha > -1/2$, we have the equivalence
$$
d_n(\alpha ) \sim {(n + \alpha + 1)^{n + \alpha + 1} \over e^{n + \alpha + 1}}.
$$
\end{lemma}

An important tool in the setting of linear dynamics is the Universality Criterion. There are several versions of it, see for instance \cite{Gro99} or \cite{GrPe11}. The one given in Theorem \ref{Univ-Crit} is sufficient for our goals. Prior to this, it is easy to extend the notion of hypercyclicity.
A sequence of continuous linear mappings \,$T_n: X \to Y$ $(n \ge 1)$ \,between two topological vector spaces $X,\, Y$ is said to be {\it universal} whenever there is a vector $x_0 \in X$ --called universal for $(T_n)$-- such that the set $\{T_n x_0: \, n \in \N\}$ is dense in $Y$. Of course, an operator $T:X \to X$ is hypercyclic if and only if the sequence $(T^n)$ of its iterates is universal.

\begin{theorem} \label{Univ-Crit}
Assume that \,$X$ and \,$Y$ are topological vector spaces, such that \,$X$ is a Baire space and \,$Y$ is separable and metrizable.
Let \,$T_n : X \to Y$ $(n \ge 1)$ \,be a sequence of continuous linear mappings.
Suppose also that there subsets \,$X_0 \subset X$ \,and \,$Y_0 \subset Y$ \,that are respectively dense in \,$X$ \,and \,$Y$ \,satisfying that, for any pair \,$(x_0,y_0) \in X_0 \times Y_0$, there are sequences \,$\{n_1 < n_2 < \cdots < n_k < \cdots \} \subset \N$ \,and \,$(x_k) \subset X$ \,such that
$$
x_k \to 0, \,\, T_{n_k}x_0 \to 0 \hbox{ \ and \ } T_{n_k}x_k \to y_0 \hbox{ \ \,as\, \ } n \to \infty .
$$
Then \,$(T_n)$ \,is universal. In fact, the set of universal vectors for \,$(T_n)$ \,is residual {\rm (}hence dense{\rm )} in \,$X$.
\end{theorem}

Corresponding notion and criterium for frequent universal sequences will be given in Section \ref{Section-freqhc-Dunkl}.
In that section the well-known Hausdorff--Young inequality (see for instance \cite{Ka76} or \cite{rudinARC}) will be used several times.
We state a version of it in the next theorem for the sake of completeness. Recall that if \,$p \in [1,\infty ]$ \,then the
conjugate exponent of \,$p$ \,is the unique \,$q \in [1,\infty ]$ \,satisfying \,${1 \over p} + {1 \over q} = 1$.

\begin{theorem} \label{Thm-HausdorffYoung}
Consider the spaces \,$L^p([0,2\pi ])$ \,of Lebesgue integrable functions
\,$[0,2\pi ] \to \C$ \,of order \,$p$. For each \,$F \in L^p ([0,2\pi ])$, let \,$\hat{F} (n) = {1 \over 2\pi} \int_0^{2 \pi} F(t) \, e^{-int} \,dt$
$(n \ge 0)$ \,be the sequence of its Fourier coefficients. If \,$1 < p \le 2$ \,and \,$q$ \,is the conjugate exponent of \,$p$ \,then
$$
\Big(\sum_{n=0}^\infty |\hat{F} (n)|^q \Big)^{1/q} \le \Big({1 \over 2\pi} \int_0^{2\pi} |F(t)|^p \,dt \Big)^{1/p}.
$$
\end{theorem}

\section{Hypercyclic entire functions for the Dunkl operator.} \label{S-Rate}

For hypercyclicity with respect to the Dunkl operator, the rate of growth $e^r/r^{\alpha + 1 }$ turns out to be critical.
We denote \,$\R^+ := (0,+\infty )$ \, and \,$\N_0 := \N \cup \{0\}$.

\begin{theorem}\label{T-ratehc}
Let $1 \leq p \leq \infty$. Then the following holds:
\begin{itemize}
\item [\rm (a)] For any function $\varphi:\mathbb{R}_+\to\mathbb{R}_+$
with $\varphi(r)\to\infty$ as $r\to\infty$ there is a
$\Lambda _{\alpha}$-hypercyclic entire function $f$ with
\[
M_p(f,r)\leq \varphi(r)\frac{e^r}{r^{\alpha + 1}}\quad\mbox{for $r>0$
sufficiently large}.
\]
\item [\rm (b)] There is no $\Lambda _{\alpha}$-hypercyclic entire function $f$ satisfying, for some $C > 0$, that
\[
M_p(f,r)\leq C \frac{e^r}{r^{\alpha + 1}}\quad\mbox{for all $r>0$.}
\]
\end{itemize}
\end{theorem}

\begin{proof} Thanks to the elementary inequality
\[
M_p(f,r)\leq M_q (f,r) \hbox{ \,for all } r > 0 \,\,\, (1 \le p \le q)
\]
it is enough to prove part (a) for the case $p = \infty$ and part (b) for the case $p=1$.

\vskip .15cm

Let us prove (a). We can assume without loss of generality that $\varphi : \Bbb R^+ \rightarrow \Bbb R^+$ with $\varphi(r) \rightarrow  \infty$ as $r\rightarrow \infty$ is monotonic and continuous on $[0,+\infty )$ with $\varphi(0) > 0$.
Consider the following space
\[
\widetilde {X} = \left\{ f\in H(\Bbb C): \, \sup_{r > 0} \frac{M_\infty (f,r) \, r^{\alpha + 1}}{\varphi(r) \, e^r} < \infty \right\}.
\]
It is not difficult to see that \,$\widetilde X$ \,is a Banach space under the norm \,$\|f\| := \sup_{r > 0} \frac{M_\infty (f,r) \, r^{\alpha + 1}}{\varphi(r) \, e^r}$ \,that is continuously embedded in $H(\mathbb{C})$.

\vskip .15cm

In addition, we define the set $X$ as the closure of the polynomials in $\widetilde X$, and the set $X_0 = Y_{0}\subset H(\mathbb{C})$ as the collection of polynomials in $z$. Note that, trivially, $X_0$ \,is dense in \,$X$ \,and \,$Y_0$ is dense in \,$Y := H(\mathbb{C})$. Now, for $n \in \N := \{1,2,3, \dots \}$, we consider the mappings
\[
T_n : X \to Y, \, T_n = \Lambda _{\alpha}^n|_X,
\]
which are continuous, and
\[
S_n:Y_0\to X, \,\, S_n=S^n \; \; \mbox{ with \,}
S(z^{k}) = \frac{d_{k}(\alpha)}{d_{k+1}(\alpha)}z^{k+1},
\]
extended linearly to \,$Y_0$.

\vskip .15cm

Now, fix a pair of polynomials $P,Q$, that is, $(P,Q) \in X_0 \times Y_0$.
Then \,$P(z) = \sum_{k=0}^p a_k z^k$ \,and \,$Q(z) = \sum_{k=0}^q b_k z^k$
\,for certain $a_0, \dots ,a_p,b_0, \dots ,b_q \in \C$.
Take as $(n_k)$ the full sequence $\{1,2,3, \dots\}$ of natural numbers
and \,$f_n (z) := S^n Q$ (with $S^n = S \circ S \circ \cdots \circ S$, $n$ times)
for all $n \in \N$. It follows from the properties of \,$\Lambda _{\alpha}$ \,given in Section \ref{Section-preliminaries}
that $T_n f_n = Q$ (so, trivially, $T_n f_n \to Q$ as $n \to \infty$) and $T_n P \to 0$ in $H(\C )$ as $n \to \infty$
(because $\Lambda^n_\alpha (z^k) = 0$ if $n \ge k$).

\vskip .15cm

Thus the conditions of the Universality Criterion (Theorem \ref{Univ-Crit}) are
satisfied if we can show that \,$\{f_n\}_{n \ge 1}$ \,converges
to \,$0$ \,in \,$X$. By linearity, one can assume that
$Q(z) = z^k$, a monomial ($k \in \mathbb{N}_0$), in which case
\[
\lim_{n \rightarrow \infty} S_n Q(z) = \lim_{n \rightarrow \infty}
\frac{d_{k}(\alpha)}{d_{k+n}(\alpha)} z^{k+n}.
\]
Therefore, all we need to show is that
\[
\lim_{n\rightarrow \infty}\frac{z^{n}}{d_{n}(\alpha)}
\]
converges to \,$0$ \,in $X$.

\vskip .15cm

To this end, fix \,$\varepsilon >0$ \,as well as an $n \in \N$.
We choose $R>0$ such that $\varphi(r)\geq 1/\varepsilon$ for
$r\geq R$. Then we have that
$$
\sup_{r\leq R} \frac{r^{\alpha + 1}}{\varphi(r) e^r} \frac{r^{n}}{d_{n}(\alpha)} \leq \frac{{R^{\alpha +1}}}{\inf_{r>0} \varphi(r)}
\frac{R^n}{d_{n}(\alpha)} \to 0 \quad \hbox{as \, } n \to \infty .
$$
Moreover a simple calculation involving the derivative of \,$r^{n + \alpha + 1} e^{-r}$ \,shows that
\[
\sup_{r\geq
R}\frac{1}{\varphi(r)} \frac{r^{\alpha + 1}}{e^r}
\frac{r^{n}}{d_n (\alpha)}\leq \varepsilon  \frac{(n+\alpha + 1)^{n + \alpha + 1}}{e^{n+\alpha +1} \, d_n(\alpha)},
\]
and Lemma \ref{Lemma-Stirling} implies that the last term is bounded by \,$C \varepsilon$ \,for any \,$n \in \mathbb{N}$.
This shows that \,$\lim_{n\rightarrow \infty} \frac{z^{n}}{d_{n}(\alpha)}$ \,converges to \,$0$ \,in \,$X$.

\vskip .15cm

In order to prove part (b), the use of the Cauchy estimates leads us to
\[
|\Lambda_{\alpha }^{n}f(0)|= \left|\frac{f^{(n)}(0)}{n!} \right|d_{n}(\alpha ) \leq \frac{d_{n}(\alpha)}{r^{n}} M_1(f,r).
\]
Assume, by way of contradiction, that there is $C > 0$ such that \,$M_1(f,r)\leq C \frac{e^r}{r^{\alpha + 1}}$ \,for all $r > 0$.
Hence we find that
$$
|\Lambda_{\alpha }^{n}f(0)| \leq C\frac{d_{n}(\alpha )}{r^{n+\alpha +1}} e^r \le C\frac{d_{n}(\alpha )}{(n+\alpha +1)^{n+\alpha +1}}
\, e^{n+\alpha +1} \hbox{ \ for all } n \geq 1.
$$
Now Lemma 1 implies that the sequence \,$\{\Lambda_{\alpha }^{n}f(0)\}_{n \ge 1}$ \,is bounded, so that \,$f$ \,cannot be hypercyclic for $\Lambda_{\alpha }$. This is the desired contradiction.
\end{proof}

\begin{remark}
An operator \,$T:H(\C ) \to H(\C )$ \,is called a {\it weighted backward shift} if there is a sequence $\{a_n\}_{n \ge 1} \subset \C \setminus \{0\}$
such that
$$
(T\, f)(z) = \sum_{n=0}^\infty a_{n+1} c_{n+1} z^n \quad (z \in \C )
$$
provided that \,$f(z) = \sum_{n=0}^\infty c_n z^n$. We denote this operator by \,$T = B_{(a_n)}$. It is not difficult to see that \,$B_{(a_n)}$ \,is well defined and continuous if and only if \,$\sup_{n \in \N}
|a_n|^{1/n} < \infty$. In \cite{Gro00a} (see also \cite{bernal1996}) is was proved that, under the assumption \,$\sup_{n \in \N}
|a_n|^{1/n} < \infty$, the operator \,$B_{(a_n)}$ \,is hypercyclic if and only if \,$\limsup_{n \to \infty} \Big| \prod_{k=1}^n a_k \Big|^{1/n} = \infty$. This result contains MacLane's theorem as a special case because the differentiation operator $D$ is the weighted backward shift \,$B_{(a_n)}$ \,with \,$a_n = n$ $(n \in \N )$. Now, observe that, for $\alpha > -1/2$, the Dunkl operator is \,$B_{(a_n)}$ \,with \,$a_n = {d_n(\alpha ) \over d_{n-1} (\alpha )}$. In this case, from Lemma 1 one easily derives that \,$\sup_{n \in \N} |a_n|^{1 \over n} < \infty$ \,and
\,$\limsup_{n \to \infty} \Big| \prod_{k=1}^n a_k \Big|^{1/n} = \limsup_{n \to \infty} d_n(\alpha )^{1 \over n} = C \, \lim_{n \to \infty} (n + \alpha + 1)^{n + \alpha + 1 \over n} = \infty$. Then the Dunkl operator is hypercyclic (note that we get an alternative proof of this result given in \cite{BeSiTri}; as a matter of fact, a much more general result is shown in \cite[Theorem 4.1]{BeSiTri}). Moreover, Grosse-Erdmann proved in \cite[Theorems 1--2]{Gro00a} that, under the assumption that $(|a_n|)$ is nondecreasing, the critical rate of growth of $M_\infty (f, \cdot )$ allowed for a $B_{(a_n)}$-hypercyclic entire function \,$f$ \,is \,$\mu (r) = \max_{n \ge 0} |r^n/\prod_{k=1}^n a_k|$. If $B_{(a_n)} = \Lambda_\alpha$ then $(|a_n|)$ is nondecreasing and $\mu (r) = \max_{n \ge 0} |r^n/d_n(\alpha )|$, so the cited theorems of \cite{Gro00a} are at our disposal. Nevertheless, we have opted for a more direct proof, which in turn yields critical rates for all $M_p(f, \cdot \,)$ $(1 \le p \le \infty )$.
\end{remark}

\section{Frequently hypercyclic entire functions for the Dunkl operator.} \label{Section-freqhc-Dunkl}

First of all, it is easy to extend the notion of frequent hypercyclicity to a sequence of mappings, see \cite{BoGro}.

\begin{definition}\label{D-fruniv}
Let $X$ and $Y$ be topological spaces, and let \,$T_n:X \to Y$ $(n \in \N )$ \,be a sequence of continuous linear mappings.
Then an element $x \in X$ is called {\it frequently universal} for
the sequence $(T_n)$ if, for every non-empty open subset $U$ of $Y$,
\[
\underline{\mbox{dens}}\{n\in \mathbb{N} : T_nx\in U\}>0.
\]
The sequence $(T_n)$ is called {\it frequently universal} if it
possesses a frequently universal element.
\end{definition}

Of course, an operator \,$T:X \to X$ \, on a topological vector space $X$ is frequently hypercyclic if
and only if the sequences of its iterates $(T^n)$ is frequently universal.
In \cite{BoGro}, a {\it Frequent Universality Criterion} was obtained
that generalizes the Frequent Hypercyclicity Criterion of Bayart
and Grivaux \cite{BaGr06}. We state it here (see Theorem \ref{T-FUC} below) only for Fr\'echet
spaces. Recall that a collection of series \,$\sum_{k=1}^\infty x_{k,j}$ $(j \in I)$ \,in a Fr\'echet space $X$ is said to be {\it unconditionally
convergent, uniformly in $j\in I$,} if for every continuous
seminorm \,$p$ \,on \,$X$ \,and every \,$\varepsilon > 0$ \,there is some
\,$N \geq 1$ such that for every finite set \,$F\subset \mathbb{N}$
with \,$F \cap \{1,2,\ldots,N\} = \varnothing$ \,and every \,$j \in I$ \,we have
that \,$p(\sum_{k\in F} x_{k,j})<\varepsilon$.

\begin{theorem} \label{T-FUC}   
Let $X$ be a Fr\'echet space, $Y$ a separable Fr\'echet space, and
\,$T_n : X\to Y$ $(n \in \mathbb{N})$ \,a sequence of operators. Suppose that
there are a dense subset $Y_0$ of \,\,$Y$ and mappings \,$S_n:Y_0\to X$
$(n \in \mathbb{N})$ such that, for all \,$y \in Y_0$,
\begin{itemize}
\item[\rm (i)] $\sum_{n=1}^k T_{k}S_{k-n}y$ converges
unconditionally in $Y$, uniformly in $k\in \mathbb{N}$, \item[\rm
(ii)] $\sum_{n=1}^\infty T_kS_{k+n}y$ converges unconditionally in
$Y$, uniformly in $k\in \mathbb{N}$, \item[\rm (iii)]
$\sum_{n=1}^\infty S_{n}y$ converges unconditionally in $X$,
\item[\rm (iv)] $T_nS_ny \to y$ \,as \,$n \to \infty$.
\end{itemize}
Then the sequence $(T_n)$ is frequently universal.
\end{theorem}

We note that the sums in (i) can be understood as infinite series by adding zero terms.

\vskip .15cm

As an auxiliary result we shall need the following
estimate. In the following, we shall adopt the convention
\,${1 \over 2p} = 0$ \,for \,$p = \infty$.

\begin{lemma}\label{L-Bar}
Let \,$1 \le q \leq 2$ \,and \,$p$ \,be the conjugate exponent of \,$q$.
Then there is some \,$C > 0$ \,such that, for all \,$r>0$, we have
\[
\sum_{n=0}^{\infty} \frac{r^{qn}}{(d_n (\alpha ))^q} \le
C \left( \frac{e^r}{r^{\alpha + {1 \over 2} + {1 \over 2p}}} \right)^q.
\]
\end{lemma}
\begin{proof}
Consider the entire functions
\[
E_{\alpha}(z;\theta,\beta)=
\sum_{n=0}^\infty\frac{z^n}{(n+\theta)^{\beta}\Gamma(\alpha n+1)}
\]
for $\alpha, \theta>0, \beta \in \mathbb{R}$. In the special case of
$\beta = 0$ these are the Mittag-Leffler functions, see for instance \cite[Section 18.1]{Er55}.
Barnes \cite[pp. 289--292]{Bar06} studied the functions \,$E_{\alpha}(z;\theta,\beta)$ and, by using
Stirling's formula, he derived asymptotical expansions from which one deduces that, for $0 < \alpha \le 2$,
\begin{equation}\label{eq1}
E_{\alpha}(r;\theta,\beta) = \alpha^{\beta-1}r^{-\beta/\alpha}
e^{r^{1/\alpha}}\big(1 +O(r^{-1/\alpha})\big) \hbox{ \ as \ } r \to +\infty .
\end{equation}

Now, by the definition of \,$d_n (\alpha )$, we get
$$
\sum_{n=0}^\infty \frac{r^{q n}}{(d_{n}(\alpha ))^q} \leq \sum_{n=0}^\infty \frac{(\Gamma(\alpha +1))^q r^{2qn}}{(2^{2n} \, n! \,
\Gamma (n + \alpha + 1))^q}
+ \sum_{n=0}^\infty \frac{(\Gamma(\alpha +1))^q r^{(2n + 1)q}}{(2^{2n+2} \, n! \, \Gamma (n + \alpha + 2))^q}
$$
$$
=(\Gamma (\alpha + 1))^q \Big(\sum_{n=0}^\infty \frac{r^{2qn}}{2^{2qn} \, n!^q \,(\Gamma(n + \alpha +1))^q}
+ \frac{r^{q}}{2^{2q}} \, \sum_{n=0}^\infty \frac{ r^{2nq}}{(2^{2qn} \, n!^q \, (\Gamma(n + \alpha +2))^q}\Big)
$$

Thanks to \eqref{eq1}, we obtain for \,$r > 0$ \,that
\begin{equation*}
\begin{split}
\sum_{n=0}^\infty \frac{ r^{2qn}}{2^{2qn} \, n!^q \, (\Gamma(n +\alpha + 1))^q} &\le C \, \sum_{n=0}^\infty \frac{((qr)^{2q})^{n}}{(n+1)^{q(\alpha +1)-1/2}\Gamma(2qn+1)} \\
&\le C \, \big((qr)^{2q}\big)^{-\frac{q(\alpha +1) - {1 \over 2}}{2q}} e^{\big((qr)^{2q}\big)^{1 \over 2q}} \\
&\le C \, \left( \frac{e^r}{r^{\alpha + {1 \over 2} + {1 \over 2p}}} \right)^q,
\end{split}
\end{equation*}
where the constant \,$C$ \,is not necessarily the same in each occurrence.

\vskip .15cm

Analogously, we have for all \,$r > 0$ \,that
$$
\frac{r^q}{2^{2q}} \sum_{n=0}^\infty \frac{ r^{2nq}}{(2^{2qn} \, n!^q \, (\Gamma (n + \alpha + 2))^q}
\le C \, r^q \, \left( \frac{e^r}{r^{\alpha + 1 + {1 \over 2} + {1 \over 2p}}} \right)^q
\le C \, \left( \frac{e^r}{r^{\alpha + {1 \over 2} + {1 \over 2p}}} \right)^q .
$$
Finally, it is enough to add up both inequalities to get the desired result.
\end{proof}

Our first main result in this section gives growth rates for which
$\Lambda _{\alpha}$-frequently hypercyclic functions exist.

\begin{theorem}\label{T-RateDFHC} Let $1\leq p\leq \infty$, and put
\,$a = \alpha + \frac{1}{2} + \frac{1}{2\max\{2,p\}}$. Then, for any function
\,$\varphi:\mathbb{R}_+ \to \mathbb{R}_+$ \,with \,$\varphi(r) \to \infty$
as \,$r\to\infty$, there is an entire function \,$f$ \,with
$$
M_p(f,r)\leq \varphi(r){\frac{e^r}{r^{a}}}\quad \mbox{for $r>0$
sufficiently large}
$$
that is frequently hypercyclic for the Dunkl operator.
\end{theorem}

\begin{proof} Since
\[
M_p(f,r)\leq M_{2}(f,r)\quad \mbox{for $1\leq p< 2$,}
\]
we need only prove the result for \,$p \geq 2$.

\vskip .15cm

Thus let $2 \leq p \leq \infty$. We shall make use of the Frequent
Universality Criterion (Theorem \ref{T-FUC}). Assuming without loss of generality that
$\inf_{r>0} \varphi(r)>0$, we consider the vector space
\[
X := \Big\{ f \in H(\mathbb{C}) : \, \sup_{r>0} \frac{M_p(f,r) \,
r^{\alpha +\frac{1}{2}+ {1 \over 2p}}}{\varphi (r) \, e^r} < \infty \Big\}.
\]
If we endow \,$X$ \,with the norm \,$\|f\| :=\sup_{r>0} \frac{M_p(f,r)
r^{\alpha +\frac{1}{2}+ {1 \over 2p}}}{\varphi (r) \, e^r}$ \,then it is not difficult to see that \,$(X,\|\cdot\|)$ \,is a Banach
space that is continuously embedded in \,$H(\mathbb{C})$.

\vskip .15cm

Let $Y_0 \subset H(\mathbb{C})$ be  the set of
polynomials, and we consider the mappings
\[
T_n = \Lambda _{\alpha}^n|_X : X \to H(\mathbb{C}),
\]
which are continuous, and
\[
S_n:Y_0\to X, \,\,S_n=S^n \; \; \mbox{ with }
S(z^{k}) = \frac{d_{k}(\alpha)}{d_{k+1}(\alpha)}z^{k+1}.
\]
Then we have for any polynomial $f$ and any $k \in \Bbb N$ that
\[
\sum_{n=1}^k T_{k}S_{k-n}f = \sum_{n=1}^k \Lambda _{\alpha}^n f,
\]
that converges unconditionally convergent in $H(\Bbb C)$,
uniformly for $k\in \Bbb N$, because \,\,$\sum_{n=1}^\infty \Lambda _{\alpha}^n f$
\,\,is a finite series. Moreover, we have according to the pro\-per\-ties given in Section \ref{Section-preliminaries} that
\[
T_n S_nf = f \quad \mbox{for any $n \in \mathbb{N}$},
\]
and
\[
\sum_{n=1}^\infty T_k S_{k+n}f = \sum_{n=1}^\infty S_n f.
\]
Thus the conditions (i)--(iv) in Theorem \ref{T-FUC} are
satisfied if we can show that \,$\sum_{n=1}^\infty S_n f$ \,converges
unconditionally in \,$X$, for any polynomial
\,$f(z) =z ^k$ $(k \in \mathbb{N}_0)$. In which case
\[
\sum_{n=1}^\infty S_n f(z) = \sum_{n=1}^\infty
\frac{d_{k}(\alpha)}{d_{k+n}(\alpha)} \, z^{k+n}.
\]

Therefore, it is sufficient to show that
\[
\sum_{n=1}^\infty \frac{z^{n}}{d_{n}(\alpha)}
\]
converges unconditionally in \,$X$. To this end, let \,$\varepsilon >0$ \,and
\,$N \in \mathbb{N}$. By Theorem \ref{Thm-HausdorffYoung} we obtain for any finite set \,$F \subset
\mathbb{N}$ \,that
\[
M_p\Big(\sum_{n\in F} \frac{z^n}{d_{n}(\alpha)},r\Big) \le \Big( \sum_{n\in
F}\frac{r^{qn}}{(d_{n}(\alpha))^{q}}\Big)^{1/q},
\]
where \,$q$ \,is the conjugate exponent of \,$p$. Hence, if
\,$F \cap \{0,1,\ldots,N\} = \varnothing$, then
\[
\Big\|\sum_{n\in F} \frac{z^n}{d_{n}(\alpha)}\Big\|  \le \Big(\sup_{r>0}
\frac{r^{q((\alpha +\frac{1}{2})+ {1 \over 2p})}}{\varphi(r)^{q}e^{qr}}\sum_{n
> N}\frac{r^{qn}}{(d_{n}(\alpha))^{q}}\Big)^{1/q}.
\]

We choose \,$R > 0$ \,such that \,$\varphi(r)^q\geq 1/\varepsilon$ \,for
\,$r\geq R$. Then we have that
\[
\sup_{r\leq R} \frac{r^{q((\alpha +\frac{1}{2}) + {1 \over 2p})}}{\varphi(r)^{q} \, e^{qr}} \, \sum_{n
> N}\frac{r^{qn}}{(d_{n}(\alpha))^{q}} \leq\frac{{R^{q((\alpha +\frac{1}{2}) + {1 \over 2p})}}}{\inf_{r > 0} \varphi(r)^{q}}
\sum_{n>N}\frac{R^{qn}}{(d_{n}(\alpha))^{q}} \longrightarrow 0
\]
as $N \to \infty$. Moreover, Lemma \ref{L-Bar} implies that
\[
\sup_{r\geq
R}\frac{1}{\varphi(r)^{q}}\frac{r^{q((\alpha +\frac{1}{2})+ {1 \over 2p})}}{e^{qr}}\sum_{n
> N}\frac{r^{qn}}{(d_{n}(\alpha))^{q}}\leq C\varepsilon \quad\mbox{for any $N \in \mathbb{N}$,}
\]
where \,$C$ \,is a constant only depending on \,$q$.

\vskip .15cm

This shows that
\[
\Big\|\sum_{n\in F}  \frac{z^{n}}{d_{n}(\alpha)}\Big\|^q \leq
(1+C)\, \varepsilon,
\]
if $\min F> N$ and $N$ is sufficiently large, so that
$\sum_{n=1}^\infty  \frac{z^{n}}{d_{n}(\alpha)}$ converges unconditionally in
\,$X$, as required.
\end{proof}

The following result, that concludes the paper, gives lower estimates on the possible growth rates.

\begin{theorem}\label{T-RateDFHC2}
Let \,$1\leq p\leq \infty$, and put \,$a = \alpha +\frac{1}{2} + \frac{1}{2\min\{2,p\}}$.
Assume that \,$\psi:\mathbb{R}_+\to\mathbb{R}_+$ \,is a function with
\,$\psi(r)\to 0$ \,as \,$r \to\infty$. Then there is no \,$\Lambda _{\alpha}$-frequently
hypercyclic entire function $f$ that satisfies
\begin{equation}\label{eq2}
M_p(f,r)\leq \psi(r){\frac{e^r}{r^{a}}}\quad \mbox{for $r>0$ sufficiently large}.
\end{equation}
\end{theorem}

\begin{proof} First, for \,$p=1$ \,the result follows immediately from Theorem \ref{T-ratehc}(b)
(notice that one may even take \,$\psi(r)\equiv C$ \,here). Moreover, since
\[
M_2(f,r)\leq M_p(f,r)\quad \mbox{for $2< p\leq\infty$,}
\]
it suffices to prove the result for \,$p\leq 2$.

\vskip .15cm

Thus let \,$1 < p\leq 2$ and $q$ is the conjugate exponent of \,$p$. We obviously may assume that \,$\psi$ \,is decreasing. Suppose that \,$f$ \,satisfy \eqref{eq2}. With
the help of Theorem \ref{Thm-HausdorffYoung} we get that
\begin{equation*}
\Big( \sum_{n=0}^\infty \Big(\frac{|f^{(n)}(0)|}{n!}r^{n}\Big)^q
\Big)^{1/q} \leq M_p(f,r)\leq \psi(r)\frac{e^r}{r^{\alpha +\frac{1}{2} + {1 \over 2p}}}
\end{equation*}
for \,$r > 0$ \,sufficiently large. Thus we have that, for large \,$r$,
\begin{equation}\label{eq3}
\sum_{n=0}^\infty \Big( \frac{|f^{(n)}(0)|}{n!}d_{n}(\alpha )\Big)^q
\cdot \frac{r^{qn+ {q \over 2p} + (\alpha +\frac{1}{2})q}e^{-qr}}{\psi(r)^q (d_{n}(\alpha ))^q} \leq 1.
\end{equation}

Using Stirling's formula we see that the function
\[
g(r) := \frac{r^{qn + {q \over 2p} +(\alpha +\frac{1}{2})q}e^{-qr}}{(d_{n}(\alpha))^q}
\]
has its maximum at \,$a_n := n + {1 \over 2p} + \alpha +\frac{1}{2}$ \,of order \,$1/\sqrt{n}$ \,and an
inflection point at \,$b_n := a_n + \sqrt{\frac{a_{n}}{q}}$. On \,$I_n:=[a_n,b_n]$, $g$ \,therefore
dominates the linear function \,$h$ \,that satisfies \,$h(a_n) = g(a_n), \,h(b_n)=0$.

\vskip .15cm

Now let \,$m \in \N$. If \,$m$ \,is sufficiently large and \,$m < n \leq 2m$
\,then \,$I_n \subset [m,3m]$. Hence we have for these \,$n$ \,that
\begin{equation*}
\begin{split}
\int_m^{3m}\frac{r^{q \, n + {q \over 2p} + (\alpha +\frac{1}{2})q} \, e^{-qr}}{\psi(r)^q(d_{n}(\alpha))^q} \, dr &\geq
\int_{I_n} \frac{h(r)}{\psi(r)^q} \, dr \\
&\geq C \frac{1}{\psi(m)^q}\frac{1}{\sqrt{n}}\sqrt{\frac n q + \frac
{{1 \over 2p}+ \alpha +\frac{1}{2}}{q}} \\
&\geq C\frac{1}{\psi(m)^q}.
\end{split}
\end{equation*}
Now, integrating \eqref{eq3} over \,$[m,3m]$ \,we obtain that for \,$m$ \,sufficiently large
\[
\frac{1}{m}\sum_{n=m+1}^{2m} \Big( \frac{|f^{(n)}(0)|}{n!}d_{n}(\alpha)\Big)^q \leq C \, \psi(m)^q.
\]
Hence
\[
\frac{1}{m}\sum_{n=0}^{m} \Big( \frac{|f^{(n)}(0)|}{n!}d_{n}(\alpha)\Big)^q \longrightarrow 0 \hbox{ \ as \ } m \to \infty .
\]
Therefore we have
\begin{align*}
\underline{\mbox{dens}}\{ n\in \mathbb{N} : \, |\Lambda_{\alpha}^{n}f(0)|>1\} &= \liminf_{m\to\infty} \frac {1}{ m} \# \{n\leq m :  |\Lambda_{\alpha}^n f(0)|>1\}\\
&\leq  \liminf_{m\to\infty} \frac{1}{m}\sum_{n=0}^{m}
\Big( \frac{|f^{(n)}(0)|}{n!}d_{n}(\alpha)\Big)^q = 0.
\end{align*}
If we take \,$U := (\delta_0)^{-1} (\{z \in \C : \, |z| > 1\})$ (where \,$\delta_0$ \,represents the $0$-evaluation functional \,$g \in H(\C ) \mapsto g(0) \in \C$, that is continuous) and \,$T := \Lambda_\alpha$, then \,$U$ \,is a nonempty open subset of \,$H(\C )$ \,and the last display shows that
\,$\underline{\mbox{dens}} \{ n \in \mathbb{N} : \, T_n f \in U\} = 0$, which prevents \,$f$ \,to be
frequently hypercyclic for the Dunkl operator. This proves the theorem.
\end{proof}

\vskip .20cm

\noindent {\bf Acknowledgements.} The first author is partially supported by the Plan Andaluz de Investigación de la Junta de Andalucía FQM-127 Grant P08-FQM-03543 and by MEC Grant MTM2012-34847-C02-01. The second author is partially supported by MEC and FEDER, project no.~MTM2014-52376-P.

\medskip

{\scriptsize
$$
\begin{array}{lr}
\mbox{Luis Bernal-Gonz\'alez } & \mbox{ Antonio Bonilla } \\
\mbox{Departamento de An\'alisis Matem\'atico } & \mbox{ Departamento de An\'alisis Matem\'atico } \\
\mbox{Universidad de Sevilla } & \mbox{ Universidad de La Laguna } \\
\mbox{Facultad de Matem\'aticas, Apdo.~1160 } &  \mbox{ C/Astrof\'{\i}sico Francisco S\'anchez, s/n } \\
\mbox{Avda.~Reina Mercedes, 41080 Sevilla, Spain} & \mbox{ 38271 La Laguna, Tenerife, Spain } \\  \mbox{E-mail: {\tt lbernal@us.es} } & \mbox{ E-mail: {\tt abonilla@ull.es} }
\end{array}
$$}

\end{document}